\documentclass[11pt,reqno]{amsart}
\usepackage{amsmath,amsthm,amssymb,amscd,url,enumerate,mathtools}
\usepackage{graphicx}
\usepackage[margin=1in]{geometry}
\usepackage{afterpage}
\usepackage{tikz}
\usepackage{todonotes}
\usepackage[all]{xy}
\usepackage[colorlinks=true,citecolor={blue},linkcolor = {purple}]{hyperref} 
\usepackage{tabularx}
\usepackage{tabu}

\usepackage{dutchcal}

\usepackage[square]{natbib}
\setcitestyle{numbers}

\usepackage{adjustbox}
\usepackage{lipsum} 

\usepackage{comment} 

\usepackage{relsize} 

\usepackage{multicol}
\usepackage{wrapfig}

\newcommand{\TITLE}{Frobenius Finds Non-monogenic Division Fields of Abelian Varieties}
\newcommand{\TITLERUNNING}{}

\theoremstyle{plain}
\newtheorem{theorem}{Theorem}

\newtheorem{lemma}[theorem]{Lemma}
\newtheorem{corollary}[theorem]{Corollary}

\theoremstyle{definition}

\theoremstyle{remark}
\newtheorem{remark}[theorem]{Remark}

\newtheorem{example}[theorem]{Example}

\newtheorem{goal}[theorem]{Goal}

\numberwithin{theorem}{section}

  {\end{list}}

  {\end{list}}

\newcommand{\tightoverset}[2]{%
  \mathop{#2}\limits^{\vbox to -.5ex{\kern-1.05ex\hbox{$#1$}\vss}}}

\newcommand{\gp}{{\mathfrak{p}}}

\def\Bcal{{\mathcal B}}

\def\Ocal{{\mathcal O}}

\newcommand{\CC}{\mathbb{C}}
\newcommand{\FF}{\mathbb{F}}
\newcommand{\GG}{\mathbb{G}}

\newcommand{\QQ}{\mathbb{Q}}
\newcommand{\RR}{\mathbb{R}}

\newcommand{\ZZ}{\mathbb{Z}}

\newcommand{\tensor}{\otimes}

\newcommand{\ol}[1]{\overline{#1}}

\newcommand{\vphi}{\varphi}

\newcommand{\abs}[1]{\lvert#1\rvert}

\newcommand{\Gal}{\operatorname{Gal}}

\newcommand{\End}{\operatorname{End}}
\newcommand{\irred}{\operatorname{irred}}
\newcommand{\ord}{\operatorname{ord}}

\newcommand{\GSp}{\operatorname{GSp}}
\newcommand{\Sp}{\operatorname{Sp}}

\title[\TITLERUNNING]{\vspace*{-1.3cm} \TITLE}

\date{\today}

\author[Hanson Smith]{Hanson Smith}
\address{Department of Mathematics, University of Connecticut,
341 Mansfield Road U1009
Storrs, CT 06269-1009
USA}
\email{hanson.smith@uconn.edu}


\begin{document}
\sloppy

\begin{abstract}
Let $A$ be an abelian variety over a finite field $k$ with $|k|=q=p^m$. Let $\pi\in \text{End}_k(A)$ denote the Frobenius and let $v=\frac{q}{\pi}$ denote Verschiebung. Suppose the Weil $q$-polynomial of $A$ is irreducible. When $\text{End}_k(A)=\mathbb{Z}[\pi,v]$, we construct a matrix which describes the action of $\pi$ on the prime-to-$p$-torsion points of $A$.
We employ this matrix in an algorithm that detects when $p$ is an obstruction to the monogeneity of division fields of certain abelian varieties.
\end{abstract}

\maketitle

\section{Introduction}\label{Intro}

One motivation of arithmetic geometry is the hope to connect properties of number rings with properties of various geometric objects. Cyclotomic fields (the division fields of $\GG_m$) supply foundational examples. One particularly nice property is that a primitive $n$-torsion point of $\GG_m$ generates the ring of integers of $\QQ(\zeta_n)$. Symbolically, $\Ocal_{\QQ(\zeta_n)}=\ZZ[\zeta_n]$. Fields with rings of integers that enjoy a power integral basis are called \textit{monogenic}. One wonders whether monogeneity\footnote{The more correct noun is probably ``monogenicity,'' but ``monogeneity'' is more common in the literature.} occurs for division fields of other algebraic groups. A natural next example to consider is the $n$-division fields of an elliptic curve. This is the subject of \cite{NonMonoECs}, where the author shows that there are many division fields of elliptic curves over $\QQ$ that are not monogenic. 

The investigation at hand is a generalization 
to abelian varieties of dimension greater than one. As is often the case, we encounter phenomena that are not present in dimension one. In spite of this we are able to generalize the main result of \cite{NonMonoECs} to abelian varieties with irreducible Weil polynomials and minimal endomorphism rings (Theorem \ref{Thm: NonMonoMainThm}). 
In the process we construct a matrix which yields the action of a lift of the Frobenius at $p$ on the prime-to-$p$-torsion points of the relevant abelian variety (Theorem \ref{Thm: FrobMatrix}). We believe this will be of independent interest and utility.

\tableofcontents

We begin with a very abbreviated list of some related research for elliptic curves.
As mentioned above, \cite{NonMonoECs} is the dimension one predecessor to this work. In \cite{abdiv}, Gonz\'alez-Jim\'enez and Lozano-Robledo classify the possible abelian division fields. In the process, they show that for $n=1,2,3,4$, and 5 one can have $\QQ(E[n])=\QQ(\zeta_n)$. Hence, in these cases, $\QQ(E[n])$ is monogenic. The monogeneity of a family of fields obtained by adjoining the $x$-coordinate of a 3-torsion point is studied in \cite{GassertSmithStange}. Adelmann \cite{Adelmann} has an in-depth investigation of the splitting of primes in division fields. Cassou-Nogu\`es and Taylor have studied the monogeneity of division fields of CM elliptic curves extensively; see \cite{C-NT}, \cite{C-NTModUnits}, and \cite{C-NTMono}.

For abelian varieties, literature in the exact same vein as our work is a bit more sparse. Division fields are studied in relation to the uniform boundedness conjecture:  Two early papers in this program for CM abelian varieties are \cite{DivFieldsCM} and \cite{SilverbergTorsionAVs}. Literature related to and generalizing this work in certain cases includes \cite{ClarkXarles}, \cite{TypicallyBounding}, and \cite{TorsionofAVswithCM}. See \cite{UniformAVs} for an overview of this area. A recent paper by Bakker and Tsimerman \cite{BakkerTsimermanGeo} considers the problem geometrically in the case of real multiplication. See also \cite{UniformBoundAbelianSchemes}, \cite{TorsionNote}, and \cite{OnTorsioninAVs}.

Understanding the points and group structure of abelian varieties over finite fields has motivated a bevy of current research; see the recent papers \cite{MarsegliaSpringer}, \cite{Kedlaya2021}, \cite{vBCLPS}. For some surveys see \cite{JeffRachelSurvey} and \cite{PriesSurvey}.

Algorithmic computation receives consistent attention. A very recent paper \cite{Marseglia} develops algorithms to compute isomorphism classes of ordinary abelian varieties. The reader interested in computation of endomorphism rings can delve into \cite{EisentragerLauter}, \cite{Bisson}, and the surrounding literature. 

The arithmetic properties of division fields of abelian varieties are studied in a multitude of other contexts. A few relevant papers are \cite{SmallDivFields}, \cite{ArithmeticofDivisionFields}, \cite{GrowthDivFields}, \cite{ClassNumbersDivFields}, and \cite{FrobEigenvalues}. 

\section*{Acknowledgments}
The author would like to thank Jeff Achter for the helpful conversations and for the computation of $|\GSp_{2g}(\ZZ/n\ZZ)|$ in Appendix \ref{OrdGSp}. The author would also like to thank Jacob Moorman for the help improving the code used in the computations. Magma \cite{Magma}, SageMath \cite{Sage}, and the \href{https://www.lmfdb.org/}{LMFDB} \cite{LMFDB} were all immensely useful for this project.

\section{Endomorphism Rings of Abelian Varieties}\label{EndofAVs} 



Let $A$ be a simple abelian variety of dimension $g$ over a finite field $k$ of characteristic $p$, where $|k|=q=p^m$. Let $\pi$ denote the Frobenius endomorphism of $A$ with respect to $k$, i.e. $x\in A(k)\mapsto x^{q}\in A(k)$, and let $v=\frac{q}{\pi}$ denote Verschiebung. Denote by $f(x)$ the characteristic polynomial of $\pi$. Over $\QQ[x]$ one has $f(x)=m(x)^e$ for some irreducible $m(x)\in \ZZ[x]$ and some $e>0$, where $e\deg(f)=2g$. We will be concerned with the case where $e=1$, so $f(x)$ is irreducible of degree $2g$. 
The endomorphism ring of $A$ over $k$ is $\End_k(A)$ and we write $E=\End_k(A)\tensor\QQ$ for the endomorphism algebra of $A$. We identify $\pi$ and $v$ with algebraic integers which we abusively also denote $\pi$ and $v$. In general, $E$ is a division algebra whose center is $\QQ(\pi)$ and $[E:\QQ]=e^2[\QQ(\pi):\QQ]$. When $e=1$, then $E$ is a totally imaginary number field of degree $2g$ and $\End_k(A)$ is identified with an order in $E$ containing $\ZZ[\pi,v]$. We will be most interested in the case where $\End_k(A)=\ZZ[\pi,v]$ because it is possible to generically compute a $\ZZ$-basis for this order. 

Changing focus slightly, a \emph{Weil $q$-number} is an algebraic integer $\pi$ such that $|\vphi(\pi)|=\sqrt{q}$ for every embedding of $\vphi: \QQ(\pi)\hookrightarrow \CC$. A \emph{Weil $q$-polynomial} is a monic, even degree polynomial $f(x)\in \ZZ[x]$ such that all the roots of $f(x)$ in $\CC$ have magnitude $\sqrt{q}$. 
If $f(x)$ is the characteristic polynomial of the Frobenius endomorphism of a simple abelian variety $A$, then $f(x)$ is a Weil $q$-polynomial of degree $2g$. 
When $f(x)$ is irreducible, then $\QQ(\pi)$ is isomorphic to $E=\End(A)\tensor \QQ$. Here $E$ is a degree $2g$ CM-field: a totally imaginary number field with a totally real subfield of index 2. As mentioned above, the minimal possible endomorphism ring of $A$ is $\ZZ[\pi,v]$.

The following is part of \cite[Theorem 6.1]{Waterhouse} and helps focus our investigation.
\begin{theorem}\label{6.1}
Suppose $A$ is a elementary (simple) abelian variety defined over the prime field $\FF_p$ and that $\QQ(\pi)$ has no real embeddings. Then
$E=\QQ(\pi)$, so the Weil $p$-polynomial $f(x)$ is irreducible. Further, all orders in $E$ containing $\pi$ and $v$ are endomorphism rings of abelian varieties over $\FF_p$.
\end{theorem}

Honda-Tate theory (\cite{Honda,HondaTate}) tells us that $k$-isogeny classes of simple abelian varieties over $k$ are in bijection with $\Gal(\overline{\QQ}/\QQ)$-conjugacy classes of Weil $q$-numbers. Combining this with Theorem \ref{6.1}, we see that by considering irreducible Weil $p$-polynomials, we are investigating the set of isogeny classes of simple abelian varieties over $\FF_p$ with only the possible exception of the class corresponding to $f(x)=(x^2-p)^g$. (Depending on $g$, this class can be either simple or non-simple.) 
It is clear that the class of Weil $p$-polynomials we are considering corresponds to abelian varieties of all possible $p$-ranks with the potential exception of $p$-rank 0. When the $g\geq 3$, there are non-supersingular abelian varieties with $p$-rank $0$; \cite{LenstraOort} shows each of these non-supersingular Newton polygons occurs for a simple abelian variety over $\FF_p$. When $g=2$ \cite[Corollary 2.11]{MaisnerNart} shows the polynomial $x^4+px^2+p^2$, for example, yields a $p$-rank 0 isogeny class of simple abelian surfaces.

The following provides us with a more concrete destination:
\begin{goal} Given an abelian variety $\hat{A}$ over $\QQ$ and a prime $p$ for which $\hat{A}$ has good reduction, let $A$ denote the reduction of $\hat{A}$ at $p$. Suppose $\End(A)=\ZZ[\pi,v]$.\footnote{It may seem our goal is not lofty enough; however, Brian Conrad remarks in \href{http://virtualmath1.stanford.edu/~conrad/mordellsem/Notes/L02.pdf}{his notes}, ``The description of endomorphism rings is a hopeless mess, on par with describing all orders in the ring of integers of a number field."} We would like to construct a matrix that describes the splitting of $p$ in any division field $\QQ(\hat{A}[n])$ with $\gcd(n,p)=1$. As a further application, we would like to have an algorithm to decide whether $p$ divides the index of every monogenic order $\ZZ[\alpha]$ in the maximal order $\Ocal_{\QQ(\hat{A}[n])}$.
\end{goal}
We have restricted our goal to $\QQ$ (and hence $\FF_p$) merely for simplicity and ease of exposition. The algorithm presented will work for arbitrary number fields after making the obvious changes.

We quickly recall a definition and a couple of results from algebraic number theory that we will need. Call a prime $p$ that divides the index of every monogenic order in the maximal order a \textit{common index divisor}. Hensel \cite{Hensel1894} shows these primes are precisely the local obstructions to monogeneity: 
\begin{theorem}\label{commonindexdivisors}
Let $K$ be a number field. The prime $p$ is a common index divisor for $K/\QQ$ if and only if there is an integer $f$ such that the number of prime ideal factors of $p\Ocal_K$ with inertia degree $f$ is greater than the number of monic irreducible polynomials of degree $f$ in $\FF_p[x]$.
\end{theorem} 
A classic result of Gauss computes the number of irreducible polynomials of a given degree in $\FF_p[x]$:
\begin{lemma}\label{Thm: IrredPolys}
Let $\mu$ be the M\"obius function. The number of irreducible polynomials of degree $m$ in $\FF_p[x]$ is given by
\[\irred(m,p):=\dfrac{1}{m}\sum_{d\mid m}p^d\mu\left(\dfrac{m}{d}\right).\]
\end{lemma}


\section{An Integral Basis for the Minimal Endomorphism Ring}



Maintaining the notation of Section \ref{EndofAVs}, in this section we wish to establish a $\ZZ$-basis for $\End_k(A)$ when $f(x)$ is irreducible and $\End_k(A)=\ZZ[\pi,v]$. 
Suppose 
\[f(x) = x^{2g}+a_{2g-1}x^{2g-1}+\cdots +a_1x+a_0\]
is an irreducible Weil $q$-polynomial. Honda-Tate theory tells us that the roots of this polynomial correspond to a $k$-isogeny class of $k$-simple abelian varieties of dimension $g$. When $k=\FF_p$, the case our computations will focus on, 
Waterhouse tells us that there is an abelian variety in this isogeny class with $\End_k(A)=\ZZ[\pi,v]$.  

Before describing the basis for $\ZZ[\pi,v]$, we recall that Weil $q$-polynomials enjoy a ``$q$-symmetry'' or ``$q$-reciprocality.'' 
\begin{lemma}\label{reciprocality}
Let $f(x)$ be an irreducible Weil $q$-polynomial written as above and suppose $g>1$, then
\[x^{2g}f\left(\frac{q}{x}\right)=q^gf(x).\]
In other words, $a_0=q^g$ and $a_i=q^{g-i}a_{2g-i}$ for $i=1,\dots , g$.
\end{lemma} 

\begin{proof}
Let $\pi$ be a root and note that $\pi \in \CC-\RR$ since $g>1$. 
Writing $\pi=a+bi$, complex conjugation acts non-trivially so $\ol{\pi}=a-bi$ is also a root. Each root of $f(x)$ has norm $\sqrt{q}$, so we have $f(x)=\prod_{i=1}^g x^2+t_ix + q$ over $\RR[x]$. We compute
\[\frac{x^2}{q}\left(\left(\frac{q}{x}\right)^2+t_i\frac{q}{x}+q\right)=q+t_ix+x^2=x^2+t_ix + q.\]
The result follows.
\end{proof}

With this lemma in hand, we are able to establish a $\ZZ$-basis.
\begin{theorem}\label{basisgeneral}
Suppose $f(x)$ is irreducible, then the set $\Bcal=\{1,\pi,\dots, \pi^g,v,\dots, v^{g-1}\}$ forms a $\ZZ$-basis for $\ZZ[\pi,v]$. If $g=1$, then $\ZZ[\pi,v]=\ZZ[\pi]$.
\end{theorem}

\begin{proof}
For $g=1$, note that $\pi+a_1=-v$, so suppose $g>1$. 
We start by evaluating $f(x)$ at $\pi$:
\begin{equation}\label{Eq: evalpi}
\pi^{2g}+a_{2g-1}\pi^{2g-1}+\cdots +a_1\pi +q^g=0.
\end{equation}
Dividing equation \eqref{Eq: evalpi} by $\pi^g$, we obtain
\[\pi^g+a_{2g-1}\pi^{g-1}+\cdots+a_{g+1}\pi+a_g+\frac{a_{g-1}}{\pi}+\cdots  +\frac{a_{1}}{\pi^{g-1}}+\frac{q^g}{\pi^g}=0.\]
Using Lemma \ref{reciprocality}, we have
\begin{equation*}\label{Weilrelation}
\pi^g+a_{2g-1}\pi^{g-1}+\cdots+a_{g+1}\pi+a_g+a_{g+1}v+\cdots  +a_{2g-1}v^{g-1}+v^g=0.
\end{equation*}
We see $v^g$ is a $\ZZ$-linear combination of our proposed basis. 

Beginning with $i=1$ and successively increasing $i$ by 1, we divide equation \eqref{Eq: evalpi} by $\pi^{g-i}$ to see that $\pi^{g+i}$ is a linear combination of elements of $\Bcal$. 
Again, starting with $i=1$ and successively increasing $i$ by 1, multiplying equation \eqref{Eq: evalpi} by $\dfrac{v^i}{\pi^g}$ expresses $v^{g+i}$ as a linear combination of elements $\Bcal$. Since $\ZZ[\pi,v]$ is a free $\ZZ$-module of rank $2g$, the spanning set $\Bcal$ is minimal and hence a $\ZZ$-basis.
\end{proof}


\section{A Matrix Representing Frobenius}\label{Sec: Main}

In this section we are concerned with generalizing work of Centeleghe \cite{Centeleghe} to abelian varieties of dimension greater than one. We note \cite{DukeToth} was the first construction of such a matrix for elliptic curves, and \cite{CojocaruPapikian} makes the analogous construction for Drinfeld modules of rank 2. 

Though our specific applications only involve abelian varieties over $\QQ$, it is not any more trouble to work over an arbitrary number field. Hence, we let $\hat A$ be an abelian variety over a number field $K$ and let $\gp$ be a prime of $K$ lying above the rational prime $p$. Denote the reduction of $\hat A$ at $\gp$ by $A$ and write $k$ for the residue field $\Ocal_K/\gp\Ocal_K$. With respect to $k$, we keep the notation established in Section \ref{EndofAVs}. 

\begin{theorem}\label{Thm: FrobMatrix}
Suppose $A/k$ is an abelian variety with an irreducible Weil $q$-polynomial such that $\End_k(A)=\ZZ[\pi,v]$. Let $\ell\neq p$ be a prime, then there is a choice of basis for the $\ell$-adic Tate module $T_{\ell}(A)$ such that the matrix $\sigma_{\gp}$ described in equation \eqref{Eq: FrobMatrix} yields the action of $\pi$ on $T_{\ell}(A)$. Consequently, if $n$ is a positive integer that is prime to $p$, then there is a choice of basis for $\hat{A}[n]$ such that $\sigma_{\gp} \bmod n$ gives the action of any element of $\Gal(\ol{K}/K)$ that reduces to $\pi$ on $\hat{A}[n]$.
\end{theorem}

Before we begin the proof, we need the following lemma. For a reference see the remark at the end of Section 4 of \cite{SerreTate}.

\begin{lemma}\label{freerank1} Let $A$ be an abelian variety over $k$ with CM by a Gorenstein ring, i.e. such that $\End_k(A)\tensor \QQ$ is a CM field. For any prime $\ell\neq p$, the Tate module $T_\ell(A)$ is free of rank 1 over $\End_k(A)\tensor \ZZ_\ell$.  
\end{lemma}

With this lemma in hand, we can prove Theorem \ref{Thm: FrobMatrix}.

\begin{proof}
First we need to know $\ZZ[\pi,v]$ is Gorenstein. 
Since $f(x)$ is irreducible, the kernel of $\ZZ[\pi,v]\to \QQ(\pi)$ is trivial and the minimal central order 
(see \cite[Definition 2]{CentelegheStix}) is $\ZZ[\pi,v]$.
The degree of $\pi$ is $2g$, so Theorem 11 of \cite{CentelegheStix} states that $\ZZ[\pi,v]$ is Gorenstein.
Therefore Lemma \ref{freerank1} shows that, for each prime $\ell\neq p$, the Tate module $T_{\ell}(A)$ is a free $\End_k(A)\tensor \ZZ_{\ell}$-module of rank 1.

In Theorem \ref{basisgeneral}, we established that $\{1,\pi,\dots, \pi^g,v,\dots, v^{g-1}\}$ forms a $\ZZ$-basis for $\ZZ[\pi,v]$. 
As before write
\[f(x) = x^{2g}+a_{2g-1}x^{2g-1}+\cdots+a_gx^g+qa_{g+1}x^{g-1}+\cdots +a_{2g-2}q^{g-2}x^2+a_{2g-1}q^{g-1}x+q^g\]
for the Weil $q$-polynomial of $A$. Dividing $f(\pi)$ by $\pi^{g-1}$ and defining $a_{2g}$ to be $1$,
\[\sum_{i=0}^g a_{2g-i}\pi^{g-i+1}+\sum_{i=1}^g a_{g+i}q v^{i-1} =0 .\]
Hence the matrix
\setcounter{MaxMatrixCols}{20}
\begin{equation}\label{Eq: FrobMatrix}
\begin{split}
&{\color{darkgray}\begin{matrix}
& 1 & \hspace{-.1 cm} \pi & \hspace{.1 cm} \pi^2 \hspace{-.1 in} & \hspace{-.05 in} \pi^{g-2} \hspace{-.2 in} & \hspace{-.1 in} \pi^{g-1} \hspace{-.1 in} & \hspace{.1 in} \pi^{g}  & & v & v^2 & \hspace{-.05 in} v^3 & & \hspace{.1 in} v^{g-1}\\
\end{matrix}}\\
\sigma_{\gp} \ \  = \ \ &\begin{bmatrix}  
    0      & 0      & 0      & \dots & 0      &-qa_{g+1} & q & 0 & 0 & \dots & 0 \\
    1      & 0      & 0      & \dots & 0      &    -a_g  & 0 & 0 & 0 & \dots & 0 \\
    0      & 1      & 0      & \dots & 0      &-a_{g+1}  & 0 & 0 & 0 & \dots & 0 \\
    \vdots & \vdots & \ddots & \cdots & \vdots &-a_{g+i-1}& 0 & 0 & 0 & \dots & 0 \\
    0      & 0      & \dots  & 1     & 0      &-a_{2g-2} & 0 & 0 & 0 & \dots & 0 \\  
    0      & 0      & \dots  & 0     & 1      &-a_{2g-1} & 0 & 0 & 0 & \dots & 0 \\
    0      & 0      & \dots  & 0     & 0      &-qa_{g+2}   & 0 & q & 0 & \dots & 0 \\  
    0      & 0      & \dots  & 0     & 0      &-qa_{g+3}   & 0 & 0 & q & \dots & 0 \\
    0      & 0      & \dots  & 0     & 0      &-qa_{g+i+1} & 0 & 0 & 0 & \ddots & 0 \\
    0      & 0      & \dots  & 0     & 0      &-qa_{2g-1} & 0 & 0 & 0 & \dots & q \\ 
    0      & 0      & \dots  & 0     & 0      &-q       & 0 & 0 & 0 & \dots  & 0 \\  
\end{bmatrix}
{\color{darkgray}\begin{matrix}
1\\
\pi\\
\pi^2\\[.1cm]
\pi^i\\
\pi^{g-1}\\
\pi^{g}\\
v\\
v^2\\[.2 cm]
v^i\\[.1 cm]
v^{g-2}\\
v^{g-1}\\
\end{matrix}}
\end{split}    
\end{equation}
yields the action of $\pi$ on $\End_k(A)=\ZZ[\pi,v]$. The gray column and row of basis elements is included for clarity. To alleviate any ambiguity, when $g=2$, we take the last row, indexed by {\color{darkgray}$v^{g-1}$}. 
Since $T_{\ell}(A)$ is free of rank 1 over $\ZZ_{\ell}[\pi,v]$ for each prime ${\ell}\neq p$, if $\gcd(n,p)=1$, then reducing $\sigma_{\gp}$ modulo $n$ yields the action of $\pi$ on the $n$-torsion points of $A$ after a suitable choice of basis for the $n$-torsion. \end{proof}


\begin{remark}\label{AVsofThm} 
It is prudent to ask whether there are actually any $A$ that satisfy the hypotheses of Theorem \ref{Thm: FrobMatrix}. Waterhouse's Theorem \ref{6.1} is useful here. If $k=\FF_p$ and $\pi\neq \pm\sqrt{p}$, then there is always an $A$ satisfying our theorem. Indeed \cite[Theorem 6.1.3]{Waterhouse} actually shows that the ideal class group of $\ZZ[\pi,v]$ acts freely and transitively on isomorphism classes of abelian varieties over $\FF_p$ with endomorphism ring equal to $\ZZ[\pi,v]$. We note that the case where $A$ is a supersingular elliptic curve, $k=\FF_p$, and $\pi =\pm \sqrt{p}$ has already been investigated in \cite{NonMonoECs}.
\end{remark}

Suppose now that the endomorphism ring of $A$ is an order $R$ of $\QQ(\pi)$ such that $\ZZ[\pi,v]\subsetneq R$. If $\ell$ is a prime that does not divide the index $[R:\ZZ[\pi,v]]$, then $\ZZ[\pi,v]\tensor \ZZ_\ell=R\tensor \ZZ_\ell=\Ocal_{\QQ(\pi)}\tensor \ZZ_\ell$. Hence Lemma \ref{freerank1} applies and the reduction of $\sigma_{\gp}$ modulo $\ell$ yields the action of the $\pi$ on the $\ell$-torsion points of $A$. We see the natural generalization holds for $n$-torsion points when $n$ is prime to the index $[R:\ZZ[\pi,v]]$.   

In summary, if $\hat{A}$ is any abelian variety over a number field $K$ whose reduction at $\gp$ is $A$ and $\tau\in \Gal(\ol{K}/K)$ is any element that reduces to $\pi_{\gp}$ modulo $\gp$, then we have described the action of $\tau$ on the prime-to-$p$ torsion points of $\hat{A}$. Immediately, we obtain a criterion for a prime $\gp$ not to split in $K(\hat{A}[n])$.

\begin{corollary}\label{NoSplit} If 
$\End_k(A)=\ZZ[\pi,v]$, then $\gp$ does not split completely in $K(\hat{A}[n])$ for any $n$ prime to $p$. 
\end{corollary}

\section{Non-monogenic Division Fields and Examples}\label{Sec: NonMonoMain}

In this section we detail how we apply the matrix constructed in Theorem \ref{Thm: FrobMatrix}. The following gives an explicit application. 

\begin{theorem}\label{NonMonoDim2MainThm}
Let $A/\FF_2$ be an abelian surface and write the Weil 2-polynomial of $A$ as 
\[x^4+a_3x^3+a_2x^2+2a_3x+2^2.\] 
Suppose the Weil 2-polynomial is irreducible, $\End_{\FF_2}(A)$ is minimal, and
\[\rho_{\hat{A},n}:\Gal(\QQ(\hat{A}[n])/\QQ)\to \GSp_{4}(\ZZ/n\ZZ)\]
is surjective for some $\hat{A}/\QQ$ that reduces to $A$ modulo $2$. Then Table \ref{p=2Dim=2} shows the $n<1000$ for which the prime 2 is a common index divisor of $\QQ(\hat{A}[n])$ over $\QQ$. 
\end{theorem}

More generally, we have 

\begin{theorem}\label{Thm: NonMonoMainThm}
Let $A/\FF_p$ be an abelian variety and write the Weil $p$-polynomial of $A$ as 
\[x^{2g}+a_{2g-1}x^{2g-1}+\cdots +p^{g-1}a_{2g-1}x+p^g.\] 
Suppose the Weil $p$-polynomial is irreducible, $\End_{\FF_p}(A)$ is minimal, and
\[\rho_{\hat{A},n}:\Gal(\QQ(\hat{A}[n])/\QQ)\to \GSp_{2g}(\ZZ/n\ZZ)\]
is surjective for some $\hat{A}/\QQ$ that reduces to $A$ modulo $p$. Then equation \eqref{Eq: MainInequality} computes the $n$ for which $p$ is a common index divisor of $\QQ(\hat{A}[n])/\QQ$. 
\end{theorem}



\begin{proof}
Let $\ord_n(\sigma_p)$ denote the order of $\sigma_p$ in $\GSp_{2g}(\ZZ)$. Since $\QQ(\hat{A}[n])$ is Galois and $p$ is unramified, Theorem \ref{commonindexdivisors} tells us $p$ is a common index divisor if and only if 
\[\frac{\left|\GSp_{2g}(\ZZ/n\ZZ)\right|}{\ord_n\left(\sigma_p\right)}>\irred\left(\ord_n\left(\sigma_p\right),p\right).\]
Using equation \eqref{Eq: sizeGSp} to compute $|\GSp_{2g}(\ZZ/n\ZZ)|$ and canceling $\ord_n(\sigma_p)$, we obtain the explicit and computable inequality
\begin{equation}\label{Eq: MainInequality}
\prod_{\ell\mid n, \ \ell \text{ prime}}\left((\ell-1)\ell^{g^2}\prod_{i=1}^g\left(\ell^{2i}-1\right)\right)\cdot\left(\ell^{2g^2+g+1}\right)^{v_\ell(n)-1}
>\sum_{\ \ d\mid \ord_n\left(\sigma_p\right)}p^d\mu\left(\dfrac{\ord_n\left(\sigma_p\right)}{d}\right).
\end{equation}
\end{proof}

\begin{remark}
In contrast to the case of elliptic curves, quite a bit less in known about the Galois representations of abelian varieties of dimension greater than 1. One does know that the open image theorem holds for abelian varieties of odd dimension and dimensions 2 and 6. See Serre's Corollaire au Th\'eor\`eme 3 of the Lettre \`a Marie-France Vign\'eras du 10/2/1986 in \cite{SerreLetter}. Mumford gave a non-open example in dimension 4. Thus, unlike our work with elliptic curves, we are not able to make statements regarding the infinitude or density of the non-monogenic division fields we have identified.  
\end{remark}


\begin{example}\label{Ex: Dim3Hyper}
Consider the hyperelliptic curve given by $y^2+y = x^7+x^6+1$ over $\FF_2$. The Jacobian of this curve $J$ is an abelian threefold with Weil 2-polynomial $x^6 - 2 x^5 + 2 x^4 - 2 x^3 + 4 x^2 - 8 x + 8$ and $2$-rank 0. In this case $\ZZ[\pi,v]$ is not the maximal order, but has index 2 in the maximal order. Thus $\End_{\FF_2}(J)$ may be bigger than $\ZZ[\pi,v]$; however, $\ZZ[\pi,v]\tensor \ZZ_{\ell}=\End_{\FF_2}(J)\tensor\ZZ_\ell$ for all $\ell\neq 2$. 

With Magma, we compute that $\FF_{2^{20}}$ is the smallest extension over which $J$ attains full 3-torsion; i.e., $J(\FF_{2^{20}})[3]\cong (\ZZ/3\ZZ)^6$. We confirm that this agrees with the order of $\sigma_2$ in $\GSp_6(\ZZ/3\ZZ)$. Thus if $\hat{A}$ is any abelian variety over $\QQ$ whose reduction at 2 is isomorphic to $J$, the prime 2 has inertia degree 20 in $\QQ(\hat{A}[3])$. There are $52,377$ irreducible polynomials of degree 20 in $\FF_{2}[x]$. However, $\GSp_6(\ZZ/3\ZZ)$ has order $18,341,406,720$. It seems far beyond the reach of current capabilities to compute much about a number field of degree anywhere near this. But, we can say that unless the index of $\rho_{\hat{A},3}(\Gal(\QQ(\hat{A}[3])/\QQ))$ in $\GSp_6(\ZZ/3\ZZ)$ is tremendously large, 2 is a common index divisor for $\QQ(\hat{A}[3])/\QQ$.
\end{example}



\begin{example}\label{Ex: Dim2Hyper}
Take the hyperelliptic curve given by $y^2=x^5 =2$ over $\FF_3$. The Jacobian $J$ is an abelian surface with Weil 3-polynomial $x^4+9$ and 3-rank 0. We compute that $\ZZ[\pi,v]$ has index 9 in the maximal order of $\QQ(\pi)= \QQ(\sqrt[4]{-9})$. Thus $\ZZ[\pi,v]\tensor \ZZ_{\ell}=\End_{\FF_3}(J)\tensor\ZZ_\ell$ for all $\ell\neq 3$.

Magma finds that $\FF_{3^4}$ is the smallest extension over which $J$ attains full 2-torsion. It is also the smallest extension over which $J$ attains full 5-torsion. As expected the order of $\sigma_3$ is 4 in $\GSp_4(\ZZ/2\ZZ)$, $\GSp_4(\ZZ/5\ZZ)$, and $\GSp_4(\ZZ/10\ZZ)$. There are only 18 irreducible polynomials of degree 4 in $\FF_4[x]$, but we have
\[\frac{|\GSp_4(\ZZ/2\ZZ)|}{4}=180, \quad \frac{|\GSp_4(\ZZ/5\ZZ)|}{4}=9,360,000 \ \ \ \text{ and } \ \ \ \frac{|\GSp_4(\ZZ/10\ZZ)|}{4}=6,739,200,000.\]
Thus if $\hat{A}/\QQ$ is an abelian surface whose reduction at 3 is $A$, we expect that 3 is a common index divisor for $\QQ(\hat{A}[2])/\QQ$, $\QQ(\hat{A}[5])/\QQ$, and $\QQ(\hat{A}[10])/\QQ$. Given that the difference between the number of irreducible polynomials and the size of the various symplectic groups is so large, we expect 3 to be a common index divisor in the larger 2-power and 5-power division fields as well as their composites.
\end{example}



\begin{example}\label{Ex: Dim4Hyper} 
Take the hyperelliptic curve given by $y^2+y=x^9+x^3+x$ over $\FF_2$. The Jacobian $J$ is an abelian fourfold with Weil 2-polynomial $x^8+16$ and 2-rank 0. We compute that $\ZZ[\pi,v]$ has index 256 in the maximal order of $\QQ(\pi)=\QQ(\sqrt[8]{-16})=\QQ(\zeta_{16})$, so $\ZZ[\pi,v]\tensor \ZZ_{\ell}=\End_{\FF_2}(J)\tensor\ZZ_\ell$ for all $\ell\neq 2$.

Since $\sigma_2$ has order 8 in $\GSp_8(\ZZ/17\ZZ)$, we know that $\FF_{2^8}$ is the smallest extension over which $J$ attains full 17-torsion. However, we were unable to compute $J(\FF_{2^8})$ with Magma because the author's computer does not have the memory to complete the computation. We can compute that $J(\FF_{2^4})\cong (\ZZ/17\ZZ)^4$. Being clever and using a more powerful machine with more memory allocated, one could likely compute $J(\FF_{2^8})$, but our point with this example is to show that our work is computationally applicable and effective.

There are only 30 polynomials of degree 8 in $\FF_2[x]$, but $|\GSp_8(\ZZ/17\ZZ)|\sim 10^{45}$. Were any abelian fourfold $\hat{A}/\QQ$ whose reduction at 2 is isomorphic $J$ to have a monogenic 17-division field, it would imply that the index of $\rho_{\hat{A},17}(\Gal(\QQ(\hat{A}[17])/\QQ))$ in $\GSp_{8}(\ZZ/17\ZZ)$ is extremely large. 
Conversely and concretely, the degree $[\QQ(\hat{A}[17]):\QQ]$ would need to be less than or equal to 240.

\end{example}


\section{Tables for Dimensions 2, 3, and 4}


In the computations presented in our tables, we will focus on abelian surfaces, threefolds, and fourfolds. The following theorems identify which coefficients yield Weil $p$-polynomials. We have cited the sources we referenced, but note that the dimension 3 case is also addressed in \cite{Xing3and4}.

For abelian surfaces:
\begin{lemma}[Lemma 2.1, \cite{MaisnerNart}]\label{ordASWeilpoly} Let $f(x)\in \ZZ[x]$ be a polynomial of degree four. The following are equivalent:
\begin{itemize}
\item $f(x) = \left(x-\pi\right)\left(x-\frac{q}{\pi}\right)\left(x-\pi'\right)\left(x-\frac{q}{\pi'}\right),$ with $\pi$ and $\pi'$ Weil $q$-numbers.
\vspace{.1 in}
\item $f(x)=(x^2-\beta_1x+q)(x^2-\beta_2x+q),$ with $\beta_i\in\RR$ and $|\beta_i|\leq 2\sqrt{q}$.
\vspace{.1 in}
\item  $f(x)=x^4+a_3x^3+a_2x^2+qa_3x+q^2$ with $|a_3|\leq 4\sqrt{q}$ and\\ 
$2|a_3|\sqrt{q}-2q\leq a_2\leq \frac{a_3^2}{4}+2q$.
\end{itemize}
\end{lemma}

Maisner and Nart note that the bound on $a_3$ can be improved to $2\left\lfloor 2\sqrt{q}\right\rfloor$. They also note that the above bounds result from the work in \cite{Ruck} and \cite{Xing1994}. 

For threefolds:
\begin{theorem}[Theorem 1.1, \cite{HalouiDim3}]\label{WeilDim3}
Let $f(x) = x^6+a_5x^5+a_4x^4+a_3x^3+qa_4x^2+q^2a_5x+q^3$ be a polynomial with integer coefficients. Then $f(x)$ is a Weil $q$-polynomial if and only if either 
\[f(x) = \left(x^2-q\right)^2\left(x^2+\beta x+ q\right),\]
where $\beta\in \ZZ$ and $|\beta|<2\sqrt{q}$, or the following conditions hold
\begin{enumerate}
\item $|a_5|<6\sqrt{q}$,
\item $4\sqrt{q} \cdot \abs{a_5}-9q<a_4\leq \frac{a_5^2}{3}+3q$,
\vspace{.05 in}
\item $-\frac{2a_5^3}{27}+\frac{a_5a_4}{3}+qa_5-\frac{2}{27}\left(a_5^2-3a_4 9q\right)^{\frac{3}{2}}\leq a_3 \leq -\frac{2a_5^3}{27}+\frac{a_5a_4}{3}+qa_5+\frac{2}{27}\left(a_5^2-3a_4 9q\right)^{\frac{3}{2}},$
\vspace{.05 in}
\item $-2qa_5-2\sqrt{q} a_4-2q\sqrt{q}<a_3<-2qa_5+2\sqrt{q}a_4+2q\sqrt{q}.$
\end{enumerate}
\end{theorem} 

For each irreducible polynomial that also satisfies the above, we are able to compute all $n$ up to a reasonable bound for which $p$ is a common index divisor for $\QQ(\hat{A}[n])$, where $\hat{A}$ is any abelian variety over $\QQ$ whose reduction at $p$ has the given Weil $p$-polynomial and the minimal endomorphism ring.

The following tables list some $n$ for which various small primes are common index divisors of $\QQ(\hat{A}[n])/\QQ$ for dimensions 2, 3, and 4. For dimension 4, we used the LMFDB to obtain the relevant Weil 2-polynomials.
When $\dim(\hat{A})=3$, there are 80 irreducible Weil $2$-polynomials, 348 irreducible Weil 3-polynomials, and 2,032 irreducible Weil 5-polynomials. These numbers increase with dimension, so we have we have abbreviated some tables. More data and the SageMath code used can be found on the webpage \href{https://math.colorado.edu/~hwsmith/CodeNonMonoAVs}{https://math.colorado.edu/$\sim$hwsmith/CodeNonMonoAVs}.


\begin{remark}\label{Rmk: prankandNewton}
In the case of elliptic curves, we are able to show in a precise way that supersingular primes are an obstruction to monogeneity. From the tables below we can see that a direct analogue (attempting to replace supersingularity with $p$-rank 0 or a specific shape of the Newton polygon) does not appear to hold. Instead, more coefficients of the Weil $p$-polynomial being zero leads to more common index divisors. Indeed, if all the $a_i=0$, the matrix $\sigma_p$ is only a permutation away from being a diagonal matrix. The order of a diagonal matrix in $\GSp_{2g}(\ZZ/n\ZZ)$ is bounded by $|\ZZ/n\ZZ^*|=\phi(n)$. We see the order of $\sigma_p$ is relatively small in this case and it becomes more likely that there are not enough irreducible polynomials of this degree in $\FF_p[x]$ to accommodate the splitting of $p$ in the $n^\text{th}$ division field. 
\end{remark}

\newpage 


\begin{table}[h!]
\centering
\begin{adjustbox}{width=5.4 in,center}
\begin{tabular}{ | m{.5cm} | m{.5cm}| m{1.2 cm}| m{10.5cm} | }
\hline
$a_3$ & $a_2$ & $p$-rank & non-monogenic $n$ \\
\hline
\hline
-3 & 5 & 2 & 3, 19, 31, 57, 61, 93, 171, 183, 589, 981 \\ \hline
-2 & 2 & 0 & 5, 7, 9, 13, 15, 21, 35, 37, 39, 45, 51, 61, 63, 65, 85, 91, 105, 109, 111, 117, 119, 133, 135, 153, 171, 185, 189, 195, 205, 219, 221, 241, 247, 255, 259, 273, 285, 305, 315, 325, 327, 333, 351, 357, 365, 377, 399, 455, 481, 485, 511, 513, 533, 545, 555, 565, 585, 595, 657, 663, 665, 673, 679, 703, 723, 741, 763, 765, 771, 777, 793, 819, 855, 873, 945, 949, 981, 999 \\ \hline
-2 & 3 & 2 & 7, 47 \\ \hline
-1 & -1 & 2 & 5, 9, 11, 15, 23, 37, 43, 45, 67, 111, 127, 135, 151, 185, 203, 301, 333, 555, 999 \\ \hline
-1 & 0 & 1 & 47 \\ \hline
-1 & 1 & 2 & 3, 9, 103, 127 \\ \hline
-1 & 3 & 2 & 5, 15, 59 \\ \hline
0 & -3 & 2 & 3, 5, 9, 11, 15, 23, 29, 33, 37, 45, 53, 87, 111, 135, 137, 185, 203, 233, 281, 301, 333, 555, 999 \\ \hline
0 & -2 & 0 & 3, 5, 7, 9, 11, 13, 15, 19, 21, 27, 33, 35, 39, 43, 45, 51, 57, 63, 65, 67, 73, 77, 81, 85, 91, 93, 99, 105, 109, 111, 117, 119, 129, 133, 135, 151, 153, 171, 185, 189, 195, 201, 217, 219, 221, 231, 241, 247, 255, 259, 273, 279, 285, 301, 315, 327, 331, 333, 337, 341, 351, 357, 365, 381, 387, 399, 441, 453, 455, 481, 485, 511, 513, 545, 555, 585, 595, 603, 627, 651, 657, 663, 665, 673, 679, 683, 693, 703, 723, 741, 763, 765, 771, 777, 819, 855, 873, 889, 903, 921, 945, 949, 981, 993, 999 \\ \hline
0 & -1 & 2 & 3, 17, 19, 23, 31, 57, 61, 93, 171, 183, 229, 589, 981 \\ \hline
0 & 1 & 2 & 3, 9, 17, 19, 23, 47, 57, 61, 69, 93, 171, 183, 229, 981 \\ \hline
0 & 2 & 0 & 3, 5, 7, 9, 13, 15, 19, 21, 27, 31, 35, 39, 45, 49, 51, 57, 63, 65, 73, 77, 85, 89, 91, 93, 99, 105, 109, 111, 117, 119, 127, 133, 135, 151, 153, 161, 171, 185, 189, 195, 217, 219, 221, 231, 241, 247, 255, 259, 273, 279, 285, 301, 315, 327, 331, 333, 337, 341, 351, 357, 365, 381, 387, 399, 441, 453, 455, 481, 485, 511, 513, 545, 553, 555, 585, 595, 623, 631, 651, 657, 663, 665, 673, 679, 693, 703, 721, 723, 741, 763, 765, 771, 777, 819, 855, 873, 889, 903, 945, 949, 981, 993, 999 \\ \hline
1 & -1 & 2 & 5, 7, 9, 11, 15, 37, 43, 45, 67, 79, 111, 135, 185, 203, 301, 333, 555, 999 \\ \hline
1 & 0 & 1 & 47 \\ \hline
1 & 1 & 2 & 3, 9 \\ \hline
1 & 3 & 2 & 5, 15, 59 \\ \hline
2 & 2 & 0 & 5, 7, 9, 13, 15, 21, 35, 37, 39, 45, 51, 61, 63, 65, 85, 91, 105, 109, 111, 117, 119, 133, 135, 153, 171, 185, 189, 195, 205, 219, 221, 241, 247, 255, 259, 273, 285, 305, 315, 325, 327, 333, 351, 357, 365, 377, 399, 455, 481, 485, 511, 513, 533, 545, 555, 565, 585, 595, 657, 663, 665, 673, 679, 703, 723, 741, 763, 765, 771, 777, 793, 819, 855, 873, 945, 949, 981, 999 \\ \hline
2 & 3 & 2 & 7, 47 \\ \hline
3 & 5 & 2 & 3, 19, 31, 57, 61, 93, 171, 183, 589, 981 \\ \hline
\end{tabular}
\end{adjustbox}
\vspace{.05 in}
\caption{$n<1000$ where 2 is a common index divisor in $\QQ(\hat A[n])$ with $\dim \hat{A}=2$}
\label{p=2Dim=2}
\end{table}


\begin{table}[h!]
\centering
\begin{adjustbox}{width=5.5 in,center}
\begin{tabular}{ | m{.5cm} | m{.5cm}| m{1.2cm}| m{12cm} | }
\hline
$a_3$ & $a_2$ & $p$-rank & non-monogenic $n$ \\
\hline
\hline
-4 & 8 & 2 & 2, 4, 5, 8, 10, 16, 17, 20, 32, 34, 40, 68, 73, 82, 88, 136, 140, 146, 164, 170, 194, 238, 272, 280, 328, 340, 365, 388, 451, 476, 485 \\ \hline
-3 & 5 & 2 & 2, 4, 29, 488 \\ \hline
-3 & 7 & 2 & 2, 4, 10, 19, 20, 40, 44, 80 \\ \hline
-2 & 1 & 2 & 2, 4, 13, 19, 20, 38, 40, 76, 86, 136, 140, 170, 194, 238, 272, 280, 340, 388, 476, 485 \\ \hline
-2 & 2 & 2 & 2, 4, 8, 10, 14, 28, 29, 58, 82, 164, 205, 410 \\ \hline
-2 & 4 & 2 & 2, 4, 8, 23 \\ \hline
-2 & 5 & 2 & 2, 4, 23 \\ \hline
-1 & -2 & 2 & 2, 4, 5, 8, 10, 13, 16, 20, 31, 37, 40, 62, 74, 80, 124, 140, 148, 160, 185, 224, 280, 296, 370 \\ \hline
-1 & -1 & 2 & 2, 4, 23 \\ \hline
-1 & 1 & 2 & 2, 4, 5, 80, 82, 100, 164, 200, 205, 328, 400, 410 \\ \hline
-1 & 2 & 2 & 2, 4, 8, 16 \\ \hline
-1 & 3 & 1 & 2, 4, 5, 17 \\ \hline
-1 & 5 & 2 & 2, 4, 29, 31, 62, 124 \\ \hline
0 & -5 & 2 & 2, 4, 5, 8, 10, 16, 20, 25, 37, 40, 74, 80, 115, 140, 148, 160, 185, 224, 265, 280, 296, 335, 370 \\ \hline
0 & -4 & 2 & 2, 4, 7, 8, 14, 16, 28, 32, 56, 82, 164, 205, 328, 410 \\ \hline
0 & -2 & 2 & 2, 4, 8, 11, 17, 20, 22, 34, 68, 82, 136, 140, 164, 170, 194, 238, 272, 280, 340, 388, 433, 451, 476, 485 \\ \hline
0 & -1 & 2 & 2, 4, 8, 11, 13, 16, 26, 29, 32, 52, 103, 104, 208, 416 \\ \hline
0 & 0 & 0 & 2, 4, 5, 8, 10, 13, 14, 16, 20, 25, 26, 28, 29, 32, 34, 35, 40, 41, 50, 52, 56, 58, 61, 64, 65, 68, 70, 73, 80, 82, 85, 88, 91, 100, 104, 110, 112, 121, 122, 130, 136, 140, 145, 146, 160, 164, 170, 176, 182, 193, 200, 205, 208, 220, 224, 242, 244, 260, 265, 272, 275, 280, 287, 290, 292, 305, 320, 328, 340, 364, 365, 386, 400, 410, 416, 440, 455, 464, 484, 488 \\ \hline
0 & 1 & 2 & 2, 4, 8, 11, 13, 16, 26, 29, 32, 44, 52, 88, 104, 208, 416 \\ \hline
0 & 3 & 0 & 2, 4, 7, 8, 13, 14, 16, 20, 26, 28, 35, 37, 40, 52, 56, 65, 70, 73, 74, 76, 80, 91, 104, 112, 121, 130, 133, 140, 143, 146, 148, 152, 160, 164, 182, 205, 208, 224, 242, 247, 259, 260, 266, 280, 286, 287, 292, 296, 328, 364, 365, 410, 416, 427, 433, 455, 481, 484, 488, 494 \\ \hline
0 & 4 & 2 & 2, 4, 7, 8, 14, 16, 23, 28, 32, 56, 82, 164, 205, 328, 410 \\ \hline
1 & -2 & 2 & 2, 4, 5, 8, 10, 16, 20, 31, 37, 40, 62, 74, 80, 124, 140, 148, 160, 181, 185, 224, 280, 296, 370 \\ \hline
1 & -1 & 2 & 2, 4, 23 \\ \hline
1 & 1 & 2 & 2, 4, 5, 80, 82, 100, 164, 200, 205, 328, 400, 410 \\ \hline
1 & 2 & 2 & 2, 4, 8, 16 \\ \hline
1 & 3 & 1 & 2, 4, 5, 17 \\ \hline
1 & 5 & 2 & 2, 4, 29, 31, 62, 124 \\ \hline
2 & 1 & 2 & 2, 4, 11, 19, 20, 38, 40, 76, 86, 136, 140, 170, 194, 238, 272, 280, 340, 388, 476, 485 \\ \hline
2 & 2 & 2 & 2, 4, 8, 10, 14, 28, 29, 58, 82, 164, 205, 410 \\ \hline
2 & 4 & 2 & 2, 4, 8, 23 \\ \hline
2 & 5 & 2 & 2, 4, 23 \\ \hline
3 & 5 & 2 & 2, 4, 29, 488 \\ \hline
3 & 7 & 2 & 2, 4, 10, 11, 19, 20, 22, 40, 44, 80 \\ \hline
4 & 8 & 2 & 2, 4, 5, 8, 10, 16, 17, 20, 32, 34, 40, 68, 73, 82, 88, 136, 140, 146, 164, 170, 194, 238, 272, 280, 328, 340, 365, 388, 451, 476, 485 \\ \hline
\end{tabular}
\end{adjustbox}
\vspace{.05 in}
\caption{$n<500$ where 3 is a common index divisor in $\QQ(\hat A[n])$ with $\dim \hat{A}=2$}
\label{p=3Dim=2}
\end{table}


\begin{table}[h!]
\centering
\begin{adjustbox}{width=6.3 in,center}
\begin{tabular}{ | m{.5cm} | m{.5cm}| m{1.2cm}| m{15cm} | }
\hline
$a_3$ & $a_2$ & $p$-rank & non-monogenic $n$ \\
\hline
\hline
-6 & 17 & 2 & 7, 12, 14, 21, 28, 42, 56, 84, 168, 217, 252, 434 \\ \hline
-5 & 15 & 0 & 11, 22, 24, 33, 44, 66, 71, 88, 124, 132, 142, 181, 213, 264, 284, 341, 362, 426, 429, 451, 492 \\ \hline
-4 & 8 & 2 & 2, 4, 8, 16, 42, 56, 84, 168, 377 \\ \hline
-4 & 11 & 2 & 8, 13, 26, 39, 48, 52, 72, 78, 104, 117, 122, 144, 156, 208, 234, 296, 312, 468, 481 \\ \hline
-3 & 4 & 2 & 3, 6, 9, 12, 18, 27, 31, 36, 54, 72, 93, 186, 279, 333, 444 \\ \hline
-3 & 8 & 2 & 12, 13, 19, 26, 38, 39, 42, 52, 56, 78, 84, 91, 104, 156, 168, 182, 228, 234, 247, 273, 312, 342, 364, 468, 494 \\ \hline
-2 & -1 & 2 & 13, 26, 39, 52, 72, 78, 104, 117, 156, 234, 312, 468, 481 \\ \hline
-2 & 4 & 2 & 2, 4, 6, 12, 24 \\ \hline
-1 & -4 & 2 & 7, 14, 28, 42, 56, 63, 72, 84, 126, 154, 168, 252 \\ \hline
-1 & 9 & 2 & 8, 24, 48 \\ \hline
0 & -9 & 2 & 7, 8, 14, 16, 17, 28, 42, 48, 51, 56, 63, 72, 84, 96, 112, 122, 126, 144, 168, 224, 252, 288, 336 \\ \hline
0 & -8 & 2 & 2, 3, 4, 8, 16, 17, 32, 34, 39, 41, 48, 52, 68, 78, 82, 104, 136, 156, 208, 272, 312 \\ \hline
0 & -7 & 2 & 4, 8, 12, 13, 16, 19, 24, 26, 38, 39, 42, 48, 52, 56, 76, 78, 84, 91, 96, 104, 112, 152, 156, 168, 182, 208, 224, 228, 234, 247, 273, 304, 312, 336, 342, 364, 416, 456, 468, 494 \\ \hline
0 & -5 & 0 & 3, 6, 7, 8, 9, 12, 14, 18, 21, 24, 27, 28, 31, 36, 39, 42, 48, 52, 54, 56, 62, 63, 72, 78, 84, 91, 93, 104, 112, 117, 122, 124, 126, 144, 156, 168, 182, 183, 186, 189, 208, 217, 228, 234, 248, 252, 266, 273, 279, 312, 336, 342, 364, 366, 372, 378, 399, 403, 427, 434, 456, 468, 496 \\ \hline
0 & -4 & 2 & 2, 4, 8, 9, 11, 16, 18, 22, 72, 136, 272, 341, 342 \\ \hline
0 & -3 & 2 & 4, 8, 48, 72, 144 \\ \hline
0 & -2 & 2 & 2, 3, 4, 7, 11, 14, 21, 22, 42, 84 \\ \hline
0 & -1 & 2 & 8, 12, 16, 24, 48, 96, 296, 333, 444 \\ \hline
0 & 0 & 0 & 2, 4, 6, 8, 12, 13, 16, 24, 26, 32, 39, 41, 42, 48, 52, 56, 62, 63, 72, 78, 82, 84, 91, 93, 96, 104, 112, 117, 124, 126, 136, 144, 156, 168, 182, 186, 192, 204, 208, 217, 221, 234, 248, 252, 272, 273, 279, 312, 313, 336, 364, 372, 403, 408, 416, 426, 429, 434, 442, 451, 468, 481, 492, 496 \\ \hline
0 & 2 & 2 & 2, 4, 7, 14, 21, 22, 42, 84, 217 \\ \hline
0 & 3 & 2 & 8, 19, 29, 38, 48, 72, 144 \\ \hline
0 & 4 & 2 & 2, 3, 4, 8, 9, 16, 18, 19, 22, 27, 57, 72, 136, 171, 272, 342 \\ \hline
0 & 5 & 0 & 4, 7, 8, 9, 12, 14, 18, 19, 21, 24, 28, 31, 36, 38, 39, 42, 48, 52, 56, 62, 63, 72, 76, 78, 84, 91, 93, 104, 112, 117, 124, 126, 142, 144, 156, 168, 181, 182, 186, 208, 217, 228, 234, 248, 252, 266, 273, 279, 284, 312, 336, 341, 342, 362, 364, 372, 378, 399, 403, 434, 456, 468, 496 \\ \hline
0 & 7 & 2 & 3, 6, 8, 11, 12, 13, 16, 18, 24, 26, 39, 42, 48, 52, 56, 78, 84, 91, 96, 104, 112, 156, 168, 182, 208, 224, 228, 234, 247, 273, 304, 312, 336, 342, 364, 416, 456, 468, 494 \\ \hline
0 & 8 & 2 & 2, 4, 8, 16, 17, 32, 34, 39, 41, 48, 52, 68, 78, 82, 104, 136, 156, 208, 272, 312 \\ \hline
1 & -4 & 2 & 7, 11, 14, 28, 42, 56, 63, 72, 84, 126, 154, 168, 252 \\ \hline
1 & -1 & 2 & 19, 22, 44, 88 \\ \hline
1 & 9 & 2 & 4, 8, 24, 48 \\ \hline
2 & -1 & 2 & 13, 26, 39, 52, 72, 78, 104, 117, 156, 234, 312, 468, 481 \\ \hline
2 & 4 & 2 & 2, 4, 6, 12, 24 \\ \hline
3 & 4 & 2 & 3, 6, 9, 12, 18, 27, 36, 54, 72, 93, 186, 279, 333, 444 \\ \hline
3 & 8 & 2 & 12, 13, 19, 26, 38, 39, 42, 52, 56, 78, 84, 91, 104, 156, 168, 182, 228, 234, 247, 273, 312, 342, 364, 468, 494 \\ \hline
4 & 8 & 2 & 2, 4, 8, 16, 42, 56, 84, 168, 377 \\ \hline
4 & 11 & 2 & 8, 13, 26, 39, 48, 52, 72, 78, 104, 109, 117, 122, 144, 156, 208, 234, 296, 312, 468, 481 \\ \hline
5 & 15 & 0 & 11, 22, 24, 33, 44, 66, 71, 88, 124, 132, 142, 181, 213, 264, 284, 341, 362, 426, 429, 451, 492 \\ \hline
6 & 17 & 2 & 7, 12, 14, 21, 28, 42, 56, 84, 168, 217, 252, 434 \\ \hline
\end{tabular}
\end{adjustbox}
\vspace{.05 in}
\caption{$n<500$ for selected $a_3$ and $a_2$ where 5 is a common index divisor in $\QQ(\hat A[n])$ with $\dim \hat{A}=2$}
\label{p=5Dim=2}
\end{table}


\begin{table}[h!]
\centering
\begin{adjustbox}{width=6.5 in,center}
\begin{tabular}{ | m{.5cm} | m{.5cm}| m{.7cm}| m{1.5 cm}| m{5.5cm} ||m{.5cm} | m{.5cm}| m{.7cm}| m{1.5 cm}| m{2.8cm} | }
\hline
$a_5$ & $a_4$ & $a_3$ & $p$-rank & non-monogenic $n$ & $a_5$ & $a_4$ & $a_3$ & $p$-rank & non-monogenic $n$\\
\hline
\hline
-4 & 9 & -15 & 3 & 7, 11, 23, 29, 43, 71, 87, 113, 127 & 0 & 1 & -3 & 3 & 3, 9 \\ \hline
-3 & 2 & 1 & 3 & 7, 11, 29, 43, 71, 87, 113, 127 & 0 & 1 & -1 & 3 &  \\ \hline
-3 & 5 & -7 & 3 &  & 0 & 1 & 1 & 3 &  \\ \hline
-3 & 6 & -9 & 3 & 3, 9, 27, 153 & 0 & 1 & 3 & 3 & 3, 9 \\ \hline
-2 & 0 & 3 & 3 & 107, 149 & 0 & 2 & -2 & 0 &  \\ \hline
-2 & 1 & 0 & 2 & 3, 5, 11, 55, 83 & 0 & 2 & -1 & 3 & 7 \\ \hline
-2 & 1 & 1 & 3 & 3 & 0 & 2 & 1 & 3 & 7 \\ \hline
-2 & 2 & -2 & 0 & 3, 9 & 0 & 2 & 2 & 0 &  \\ \hline
-2 & 2 & -1 & 3 & 3 & 0 & 3 & -1 & 3 & 3, 89, 153 \\ \hline
-2 & 3 & -6 & 2 & 73 & 0 & 3 & 1 & 3 & 3, 7, 89, 153 \\ \hline
-2 & 3 & -5 & 3 & 3, 9, 27, 59, 63 & 1 & -1 & -5 & 3 & 3, 9 \\ \hline
-2 & 3 & -3 & 3 & 5, 83, 131 & 1 & -1 & -4 & 2 & 3, 7, 49 \\ \hline
-2 & 4 & -6 & 0 & 3 & 1 & -1 & -2 & 2 & 3, 9 \\ \hline
-2 & 5 & -7 & 3 & 3, 7 & 1 & 0 & -3 & 3 & 7, 77, 103 \\ \hline
-1 & -1 & 2 & 2 & 3, 9 & 1 & 0 & -1 & 3 & 5, 15 \\ \hline
-1 & -1 & 4 & 2 & 3, 7 & 1 & 0 & 0 & 1 &  \\ \hline
-1 & -1 & 5 & 3 & 3, 9 & 1 & 0 & 1 & 3 & 3 \\ \hline
-1 & 0 & -1 & 3 & 3 & 1 & 1 & 0 & 2 & 3, 7 \\ \hline
-1 & 0 & 0 & 1 &  & 1 & 1 & 1 & 3 & 3 \\ \hline 
-1 & 0 & 1 & 3 & 5, 15 & 1 & 1 & 3 & 3 & 3 \\ \hline
-1 & 0 & 3 & 3 & 7, 77, 103 & 1 & 2 & 1 & 3 & 3, 5, 7, 9, 35, 75 \\ \hline
-1 & 1 & -3 & 3 & 3 & 1 & 2 & 2 & 1 & 3 \\ \hline
-1 & 1 & -1 & 3 & 3 & 1 & 2 & 3 & 3 & 3 \\ \hline
-1 & 1 & 0 & 2 & 3, 7 & 1 & 2 & 5 & 3 & 3, 5 \\ \hline
-1 & 2 & -5 & 3 & 3, 5 & 1 & 3 & 3 & 3 & 5, 25 \\ \hline
-1 & 2 & -3 & 3 & 3 & 1 & 4 & 3 & 3 & 3, 83, 127 \\ \hline
-1 & 2 & -2 & 1 & 3 & 2 & 0 & -3 & 3 & 107, 149 \\ \hline
-1 & 2 & -1 & 3 & 3, 5, 7, 9, 35, 75 & 2 & 1 & -1 & 3 & 3 \\ \hline
-1 & 3 & -3 & 3 & 5, 25 & 2 & 1 & 0 & 2 & 3, 5, 11, 55, 83 \\ \hline
-1 & 4 & -3 & 3 & 3, 83, 127 & 2 & 2 & 1 & 3 & 3 \\ \hline
0 & -1 & -2 & 2 & 73 & 2 & 2 & 2 & 0 & 3, 9 \\ \hline
0 & -1 & -1 & 3 &  & 2 & 3 & 3 & 3 & 5, 83, 131 \\ \hline
0 & -1 & 1 & 3 &  & 2 & 3 & 5 & 3 & 3, 9, 27, 59, 63 \\ \hline
0 & -1 & 2 & 2 & 73 & 2 & 3 & 6 & 2 & 73 \\ \hline
0 & 0 & -3 & 3 & 3, 7, 9, 13, 15, 21, 27, 29, 31, 35, 39, 45, 63, 65, 87, 91, 93, 105, 117, 123, 141, 151, 195 & 2 & 4 & 6 & 0 & 3 \\ \hline
0 & 0 & -2 & 0 & 3, 7, 11, 15, 23, 29, 37, 45, 67, 71, 79 & 2 & 5 & 7 & 3 & 3, 7 \\ \hline
0 & 0 & -1 & 3 & 3, 5, 7, 15, 19, 21, 25, 35, 45, 63, 71, 75, 95, 97, 105, 123, 133 & 3 & 2 & -1 & 3 & 7, 11, 23, 29, 43, 71, 87, 113, 127 \\ \hline
0 & 0 & 1 & 3 & 3, 5, 7, 15, 19, 21, 25, 35, 45, 47, 49, 63, 75, 95, 97, 105, 123, 133 & 3 & 5 & 7 & 3 & 7 \\ \hline
0 & 0 & 2 & 0 & 3, 7, 11, 15, 23, 29, 37, 45, 67 & 3 & 6 & 9 & 3 & 3, 9, 27, 153 \\ \hline
0 & 0 & 3 & 3 & 3, 7, 9, 13, 15, 21, 27, 29, 31, 35, 39, 45, 47, 63, 65, 71, 87, 91, 93, 105, 117, 123, 141, 151, 195 & 4 & 9 & 15 & 3 & 7, 11, 29, 43, 71, 87, 113, 127 \\ \hline
\end{tabular}
\end{adjustbox}
\vspace{.5 cm}
\caption{$n<200$ where 2 is a common index divisor in $\QQ(\hat A[n])$, with $\dim \hat{A}=3$}
\label{p=2Dim=3Double}
\end{table}


\begin{table}[h!]
\centering
\begin{adjustbox}{width=5.7 in,center}
\begin{tabular}{ | m{.5cm} | m{.5cm}| m{.7cm}| m{1.2 cm}| m{8cm} | }
\hline
$a_5$ & $a_4$ & $a_3$ & $p$-rank & non-monogenic $n$ \\
\hline
\hline
0 & 0 & -9 & 0 & 2, 4, 7, 8, 13, 14, 19, 26, 28, 37, 38, 40, 52, 56, 65, 70, 73, 74, 76, 80, 91, 95, 104, 112, 122, 124, 130, 133, 140, 146, 148, 152, 182, 185, 190 \\ \hline
0 & 0 & -7 & 3 & 2, 4, 5, 7, 8, 10, 11, 13, 14, 17, 19, 20, 22, 25, 26, 28, 35, 38, 40, 49, 56, 70, 76, 80, 85, 91, 112, 119, 133, 140, 145, 154, 175 \\ \hline
0 & 0 & -6 & 0 & 2, 4, 8, 11, 13, 14, 16, 17, 19, 22, 26, 28, 34, 38, 40, 43, 44, 56, 68, 86, 112, 119, 136, 172 \\ \hline
0 & 0 & -5 & 3 & 2, 4, 5, 7, 8, 11, 13, 14, 17, 23, 28, 37, 40, 56, 61, 70, 74, 80, 85, 112, 124, 140, 148, 154 \\ \hline
0 & 0 & -4 & 3 & 2, 4, 8, 13, 14, 16, 26, 28, 32, 40, 52, 56, 65, 76, 80, 95, 104, 112, 116, 124, 130, 152, 190 \\ \hline
0 & 0 & -3 & 0 & 2, 4, 5, 8, 10, 13, 20, 25, 31, 35, 40, 47, 67, 80, 124, 155 \\ \hline
0 & 0 & -2 & 3 & 2, 4, 5, 7, 8, 10, 13, 14, 16, 17, 20, 25, 28, 34, 35, 40, 47, 50, 56, 68, 70, 80, 85, 100, 109, 136, 140, 160, 170 \\ \hline
0 & 0 & -1 & 3 & 2, 4, 8, 11, 13, 19, 22, 26, 38, 40, 52, 65, 76, 79, 80, 104, 122, 130, 143, 193 \\ \hline
0 & 0 & 1 & 3 & 2, 4, 8, 11, 13, 19, 23, 26, 38, 40, 47, 52, 65, 76, 79, 80, 104, 122, 130, 143, 193 \\ \hline
0 & 0 & 2 & 3 & 2, 4, 5, 7, 8, 10, 14, 16, 17, 20, 25, 28, 34, 35, 40, 50, 56, 68, 70, 80, 85, 100, 109, 136, 140, 160, 170 \\ \hline
0 & 0 & 3 & 0 & 2, 4, 5, 8, 10, 13, 20, 23, 25, 26, 31, 35, 40, 67, 80, 124, 155 \\ \hline
0 & 0 & 4 & 3 & 2, 4, 8, 11, 13, 14, 16, 26, 28, 32, 40, 47, 52, 56, 65, 76, 80, 95, 104, 112, 116, 124, 130, 152, 190 \\ \hline
0 & 0 & 5 & 3 & 2, 4, 5, 7, 8, 11, 13, 14, 17, 22, 26, 28, 37, 40, 56, 61, 70, 74, 80, 85, 109, 112, 124, 140, 148, 154 \\ \hline
0 & 0 & 6 & 0 & 2, 4, 8, 11, 13, 14, 16, 17, 19, 22, 26, 28, 34, 38, 40, 43, 44, 56, 68, 86, 112, 119, 136, 172 \\ \hline
0 & 0 & 7 & 3 & 2, 4, 5, 7, 8, 10, 13, 14, 17, 19, 20, 25, 28, 35, 38, 40, 49, 56, 70, 76, 80, 85, 91, 112, 119, 133, 140, 145, 154, 175 \\ \hline
0 & 0 & 9 & 0 & 2, 4, 7, 8, 13, 14, 19, 26, 28, 37, 38, 40, 52, 56, 65, 70, 73, 74, 76, 80, 91, 95, 104, 112, 122, 124, 130, 133, 140, 146, 148, 152, 182, 185, 190 \\ 
\hline
\end{tabular}
\end{adjustbox}
\vspace{.05 in}
\caption{$n<200$ where 3 is a common index divisor in $\QQ(\hat A[n])$ with $\dim \hat{A}=3$ and $a_5=a_4=0$. }
\label{p=3Dim=3}
\end{table}


\begin{table}[h!]
\centering
\begin{adjustbox}{width=6.1 in,center}
\begin{tabular}{ | m{.5cm} | m{.5cm}| m{.7cm}| m{1.1 cm}| m{14cm} | }
\hline
$a_5$ & $a_4$ & $a_3$ & $p$-rank & non-monogenic $n$ \\
\hline
\hline
0 & 0 & -21 & 3 & 3, 4, 6, 7, 8, 9, 12, 14, 16, 18, 19, 21, 24, 28, 36, 38, 42, 48, 49, 51, 56, 57, 63, 72, 76, 84, 108, 112, 114, 119, 126, 133, 144, 147, 152, 168, 171, 189 \\ \hline
0 & 0 & -20 & 0 & 2, 4, 8, 12, 16, 24, 26, 42, 48, 51, 52, 56, 68, 84, 96, 102, 104, 109, 112, 119, 136, 168 \\ \hline
0 & 0 & -19 & 3 & 4, 7, 8, 11, 14, 16, 19, 26, 28, 29, 31, 37, 38, 44, 52, 56, 74, 76, 91, 104, 112, 133, 148, 152, 163, 182 \\ \hline
0 & 0 & -17 & 3 & 4, 8, 11, 13, 16, 26, 39, 52, 76, 91, 93, 104, 109, 143, 152 \\ \hline
0 & 0 & -16 & 3 & 2, 4, 7, 8, 11, 12, 14, 16, 22, 24, 26, 28, 31, 48, 52, 56, 71, 88, 91, 96, 104, 142, 164, 182 \\ \hline
0 & 0 & -15 & 0 & 3, 4, 6, 8, 9, 12, 16, 18, 24, 36, 37, 48, 72, 74, 108, 111, 144, 148 \\ \hline
0 & 0 & -13 & 3 & 4, 8, 9, 13, 16, 31, 39, 51, 63, 76, 91, 117, 119, 152, 153 \\ \hline
0 & 0 & -12 & 3 & 2, 3, 4, 6, 7, 8, 9, 12, 14, 16, 18, 19, 21, 24, 28, 31, 36, 38, 42, 48, 51, 56, 57, 62, 63, 68, 72, 76, 84, 93, 96, 102, 114, 124, 126, 133, 136, 138, 152, 168, 171, 186, 189 \\ \hline
0 & 0 & -11 & 3 & 4, 8, 16, 19, 38, 61, 76, 152 \\ \hline
0 & 0 & -10 & 0 & 2, 4, 8, 12, 24, 26, 29, 42, 48, 51, 52, 56, 58, 68, 84, 102, 112, 119, 136, 168 \\ \hline
0 & 0 & -9 & 3 & 3, 4, 6, 7, 8, 9, 12, 13, 14, 16, 18, 21, 24, 26, 27, 28, 36, 38, 39, 42, 48, 52, 54, 56, 63, 72, 78, 81, 84, 91, 104, 108, 112, 117, 123, 126, 144, 156, 162, 168, 182, 189 \\ \hline
0 & 0 & -8 & 3 & 2, 4, 8, 11, 12, 16, 22, 24, 31, 42, 48, 56, 62, 84, 88, 93, 96, 112, 124, 134, 168, 186 \\ \hline
0 & 0 & -7 & 3 & 4, 7, 8, 14, 16, 19, 21, 26, 28, 29, 31, 38, 49, 51, 52, 56, 62, 76, 91, 104, 112, 119, 124, 133, 152, 161, 163, 182 \\ \hline
0 & 0 & -6 & 3 & 2, 3, 4, 6, 8, 9, 11, 12, 18, 19, 22, 24, 26, 31, 33, 36, 39, 42, 48, 52, 56, 57, 66, 78, 84, 112, 114, 117, 132, 156, 168, 171 \\ \hline
0 & 0 & -5 & 0 & 4, 7, 8, 9, 11, 14, 16, 28, 43, 51, 56, 93, 112, 153 \\ \hline
0 & 0 & -3 & 3 & 3, 4, 6, 8, 9, 11, 12, 16, 18, 24, 26, 31, 33, 36, 38, 39, 43, 48, 51, 52, 62, 72, 76, 78, 79, 104, 108, 117, 119, 123, 124, 129, 144, 156, 181 \\ \hline
0 & 0 & -2 & 3 & 2, 4, 7, 8, 12, 14, 19, 24, 28, 48, 133, 164 \\ \hline
0 & 0 & -1 & 3 & 4, 8, 16, 29, 31, 62, 124, 127 \\ \hline
0 & 0 & 1 & 3 & 4, 8, 16, 29, 31, 62, 124, 127 \\ \hline
0 & 0 & 2 & 3 & 2, 4, 7, 8, 12, 14, 24, 28, 48, 133, 164 \\ \hline
0 & 0 & 3 & 3 & 3, 4, 6, 8, 9, 11, 12, 16, 18, 24, 26, 33, 36, 39, 43, 48, 51, 52, 72, 78, 104, 108, 117, 119, 123, 129, 144, 156, 181 \\ \hline
0 & 0 & 5 & 0 & 4, 7, 8, 9, 11, 14, 16, 28, 43, 51, 56, 93, 112, 153, 191 \\ \hline
0 & 0 & 6 & 3 & 2, 3, 4, 6, 8, 9, 11, 12, 18, 22, 24, 26, 33, 36, 39, 42, 48, 52, 56, 57, 66, 78, 79, 84, 112, 114, 117, 132, 156, 168, 171 \\ \hline
0 & 0 & 7 & 3 & 4, 7, 8, 14, 16, 19, 21, 26, 28, 29, 38, 49, 51, 52, 56, 76, 91, 104, 112, 119, 133, 152, 161, 163, 182 \\ \hline
0 & 0 & 8 & 3 & 2, 4, 8, 11, 12, 16, 19, 22, 24, 31, 42, 48, 56, 62, 84, 88, 93, 96, 112, 124, 134, 168, 186 \\ \hline
0 & 0 & 9 & 3 & 3, 4, 6, 7, 8, 9, 12, 13, 14, 16, 18, 21, 24, 26, 27, 28, 36, 39, 42, 44, 48, 52, 54, 56, 63, 72, 78, 81, 84, 91, 104, 108, 112, 117, 123, 126, 144, 156, 162, 168, 182, 189 \\ \hline
0 & 0 & 10 & 0 & 2, 4, 8, 12, 24, 26, 29, 42, 48, 51, 52, 56, 58, 68, 84, 102, 112, 119, 136, 168 \\ \hline
0 & 0 & 11 & 3 & 4, 8, 16, 61, 76, 152 \\ \hline
0 & 0 & 12 & 3 & 2, 3, 4, 6, 7, 8, 9, 12, 14, 16, 18, 19, 21, 24, 28, 31, 36, 38, 42, 48, 51, 56, 57, 62, 63, 68, 72, 76, 84, 93, 96, 102, 114, 124, 126, 133, 136, 138, 152, 168, 171, 186, 189 \\ \hline
0 & 0 & 13 & 3 & 4, 8, 9, 13, 16, 19, 31, 38, 39, 51, 62, 63, 76, 91, 117, 119, 124, 152, 153 \\ \hline
0 & 0 & 15 & 0 & 3, 4, 6, 8, 9, 12, 16, 18, 24, 31, 36, 37, 48, 62, 72, 74, 108, 111, 144, 148, 191 \\ \hline
0 & 0 & 16 & 3 & 2, 4, 7, 8, 11, 12, 14, 16, 22, 24, 26, 28, 48, 52, 56, 71, 88, 91, 96, 104, 142, 164, 182 \\ \hline
0 & 0 & 17 & 3 & 4, 8, 11, 13, 16, 19, 26, 38, 39, 44, 52, 76, 79, 91, 93, 104, 109, 143, 152 \\ \hline
0 & 0 & 19 & 3 & 4, 7, 8, 11, 14, 16, 19, 26, 28, 29, 31, 37, 38, 52, 56, 62, 74, 76, 79, 91, 104, 112, 133, 148, 152, 163, 182 \\ \hline
0 & 0 & 20 & 0 & 2, 4, 8, 12, 16, 24, 26, 42, 48, 51, 52, 56, 68, 84, 96, 102, 104, 109, 112, 119, 136, 168 \\ \hline
0 & 0 & 21 & 3 & 3, 4, 6, 7, 8, 9, 12, 14, 16, 18, 21, 24, 28, 36, 42, 48, 49, 51, 56, 57, 63, 72, 76, 84, 108, 112, 114, 119, 126, 133, 144, 147, 152, 168, 171, 189 \\ \hline
\end{tabular}
\end{adjustbox}
\vspace{.05 in}
\caption{$n<200$ where 5 is a common index divisor in $\QQ(\hat A[n])$ with $\dim \hat{A}=3$ and $a_5=a_4=0$. }
\label{p=5Dim=3}
\end{table}


\begin{table}[h!]
\centering
\begin{adjustbox}{width=6.3 in,center}
\begin{tabular}{ | m{.5cm} | m{.5cm}| m{.5cm}| m{.7cm}| m{1.1 cm}| m{12cm} | }
\hline
$a_7$ & $a_6$ & $a_5$ & $a_4$ & $p$-rank & non-monogenic $n$ \\
\hline
\hline
0 & 0 & -3 & 1 & 4 & 3, 9, 27 \\ \hline
0 & 0 & -2 & -2 & 0 & 3 \\ \hline
0 & 0 & -2 & -1 & 4 & 11 \\ \hline
0 & 0 & -2 & 0 & 0 &  \\ \hline
0 & 0 & -2 & 2 & 0 &  \\ \hline
0 & 0 & -2 & 3 & 4 &  \\ \hline
0 & 0 & -1 & -5 & 4 & 3, 29 \\ \hline
0 & 0 & -1 & -4 & 3 &  \\ \hline
0 & 0 & -1 & -3 & 4 &  \\ \hline
0 & 0 & -1 & -1 & 4 &  \\ \hline
0 & 0 & -1 & 1 & 4 & 3 \\ \hline
0 & 0 & -1 & 2 & 3 &  \\ \hline
0 & 0 & -1 & 3 & 4 & 17 \\ \hline
0 & 0 & -1 & 5 & 4 &  \\ \hline
0 & 0 & 0 & -7 & 4 & 3, 5, 7, 9, 11, 13, 15, 17, 19, 21, 23, 25, 31, 33, 39, 47, 51, 53, 55, 57, 61, 63, 85, 93 \\ \hline
0 & 0 & 0 & -6 & 0 & 3, 5, 9, 11, 15, 17, 23, 25, 27, 29, 33, 35, 37, 43, 45, 51, 53, 55, 65, 67, 69, 85, 87, 91, 99 \\ \hline
0 & 0 & 0 & -5 & 4 & 3, 5, 7, 9, 11, 15, 17, 21, 25, 27, 33, 35, 39, 41, 43, 45, 47, 51, 55, 61, 63, 79, 81, 85, 99 \\ \hline
0 & 0 & 0 & -3 & 4 & 3, 5, 7, 9, 13, 15, 17, 21, 27, 31, 35, 39, 41, 43, 49, 51, 63, 83, 89, 91, 93 \\ \hline
0 & 0 & 0 & -2 & 0 & 3, 5, 7, 9, 15, 17, 19, 21, 23, 25, 27, 31, 39, 45, 47, 51, 53, 57, 79, 83, 93 \\ \hline
0 & 0 & 0 & 0 & 0 & 3, 5, 7, 9, 11, 13, 15, 17, 19, 21, 25, 27, 29, 31, 33, 35, 37, 39, 41, 43, 45, 51, 53, 55, 57, 63, 65, 69, 73, 75, 85, 87, 89, 91, 93, 95, 97, 99 \\ \hline
0 & 0 & 0 & 1 & 4 & 3, 5, 9, 11, 13, 15, 17, 23, 25, 27, 29, 31, 33, 35, 37, 39, 43, 45, 51, 53, 55, 57, 65, 75, 85, 87, 91, 95 \\ \hline
0 & 0 & 0 & 2 & 0 & 3, 5, 7, 9, 13, 15, 17, 19, 21, 23, 31, 39, 47, 51, 53, 57, 93 \\ \hline
0 & 0 & 0 & 3 & 4 & 3, 5, 7, 9, 13, 15, 17, 21, 25, 27, 31, 35, 39, 43, 49, 51, 59, 63, 71, 89, 91, 93, 95 \\ \hline
0 & 0 & 0 & 5 & 4 & 3, 5, 7, 9, 11, 15, 17, 21, 25, 27, 29, 33, 35, 41, 43, 45, 47, 51, 55, 59, 61, 63, 71, 81, 83, 85, 99 \\ \hline
0 & 0 & 0 & 6 & 0 & 3, 5, 7, 9, 11, 13, 15, 17, 23, 25, 27, 29, 33, 35, 37, 43, 45, 51, 53, 55, 65, 69, 71, 79, 85, 87, 91, 99 \\ \hline
0 & 0 & 1 & -5 & 4 & 3, 29 \\ \hline
0 & 0 & 1 & -4 & 3 &  \\ \hline
0 & 0 & 1 & -3 & 4 &  \\ \hline
0 & 0 & 1 & -1 & 4 &  \\ \hline
0 & 0 & 1 & 1 & 4 & 3 \\ \hline
0 & 0 & 1 & 2 & 3 &  \\ \hline
0 & 0 & 1 & 3 & 4 & 17 \\ \hline
0 & 0 & 1 & 5 & 4 &  \\ \hline
0 & 0 & 2 & -2 & 0 & 3 \\ \hline
0 & 0 & 2 & -1 & 4 & 11 \\ \hline
0 & 0 & 2 & 0 & 0 &  \\ \hline
0 & 0 & 2 & 2 & 0 &  \\ \hline
0 & 0 & 2 & 3 & 4 &  \\ \hline
\end{tabular}
\end{adjustbox}
\vspace{.05 in}
\caption{$n<100$ where 2 is a common index divisor in $\QQ(\hat A[n])$ with $\dim \hat{A}=4$ and $a_7=a_6=0$ }
\label{p=2Dim=4}
\end{table}


\appendix

\section{The Size of the General Symplectic Group}\label{OrdGSp}


Let $\GSp_{2g}$ be the general symplectic group of dimension $2g$ and $\Sp_{2g}$ be the symplectic group of dimension $2g$. We expect $\QQ(\hat A[n])$ to have degree equal to $|\GSp_{2g}(\ZZ/n\ZZ)|$ over $\QQ$. Thus our algorithm requires $|\GSp_{2g}(\ZZ/n\ZZ)|$. Let $\ell\neq p$ be a rational prime. Note that 
\[|\Sp_{2g}(\ZZ/\ell\ZZ)|=\ell^{g^2}\prod_{i=1}^g\left(\ell^{2i}-1\right).\]
One can consult \cite[page 147]{Artin} for this result. From the exact sequence
\[1\to \Sp_{2g}(\ZZ/\ell\ZZ)\to \GSp_{2g}(\ZZ/\ell\ZZ)\to\GG_m(\ZZ/\ell\ZZ)\to 1,\]
we see that 
\begin{equation*}\label{GSpprime}
|\GSp_{2g}(\ZZ/\ell\ZZ)|=(\ell-1)\ell^{g^2}\prod_{i=1}^g\left(\ell^{2i}-1\right).
\end{equation*}
The Lang-Weil bound \cite{LangWeil} tells us that the dimension of $\GSp_{2g}$ is $2g^2+g+1$. Since $\GSp_{2g}$ is smooth over $\ZZ_\ell$, we have
\[|\GSp_{2g}(\ZZ/\ell^e\ZZ)|=|\GSp_{2g}(\ZZ/\ell\ZZ)|\cdot\left(\ell^{\dim g}\right)^{e-1}. \]
We obtain
\begin{equation*}
|\GSp_{2g}(\ZZ/\ell^e\ZZ)|=\left((\ell-1)\ell^{g^2}\prod_{i=1}^g\left(\ell^{2i}-1\right)\right)\cdot\left(\ell^{2g^2+g+1}\right)^{e-1}.
\end{equation*}

Finally, for composite $n$, we have
\begin{equation}\label{Eq: sizeGSp}
\begin{split}
|\GSp_{2g}(\ZZ/n\ZZ)|&=\prod_{\ell\mid n, \ \ell \text{ prime}}\left|\GSp_{2g}\left(\ZZ/p^{v_\ell(n)}\ZZ\right)\right|\\
&=\prod_{\ell\mid n, \ \ell \text{ prime}}\left((\ell-1)\ell^{g^2}\prod_{i=1}^g\left(\ell^{2i}-1\right)\right)\cdot\left(\ell^{2g^2+g+1}\right)^{v_\ell(n)-1}.
\end{split}
\end{equation}

\bibliography{Bibliography}
\bibliographystyle{alpha} 

\end{document}